\documentclass[letterpaper, 10 pt, conference]{ieeeconf}

\IEEEoverridecommandlockouts                              

\usepackage{graphics} 
\usepackage{epsfig} 
\usepackage{amsmath} 
\usepackage{amssymb}
\usepackage{setspace,color}

\title{\LARGE \bf
Sparse Linear-Quadratic Feedback Design Using Affine Approximation
}

\author{MirSaleh Bahavarnia $^{1}$
\thanks{$^{1}$The Author is with the Department of Mechanical Engineering and Mechanics, Lehigh University,
        Bethlehem, PA 18015, USA
        {\tt\small mib313@lehigh.edu}}%
}

\begin{document}

\newtheorem{theorem}{Theorem}
\newtheorem{lemma}{Lemma}
\newtheorem{problem}{Problem}
\newtheorem{example}{Example}
\newtheorem{proposition}{Proposition}
\newtheorem{corollary}{Corollary}
\newtheorem{conjecture}{Conjecture}
\newtheorem{remark}{Remark}

\maketitle
\thispagestyle{empty}
\pagestyle{empty}

\begin{abstract}

We consider a class of $\ell_0$-regularized linear-quadratic (LQ) optimal control problems. This class of problems is obtained by augmenting  a penalizing sparsity measure to the cost objective of the standard linear-quadratic regulator (LQR) problem in order to promote sparsity pattern of the state feedback controller. This class of problems is generally NP hard and computationally intractable. First, we apply a $\ell_1$-relaxation and consider the $\ell_1$-regularized LQ version of this class of problems, which is still nonconvex. Then, we convexify the resulting $\ell_1$-regularized LQ problem by applying affine approximation techniques. An iterative algorithm is proposed to solve the $\ell_1$-regularized LQ problem using a series of convexified $\ell_1$-regularized LQ problems. By means of several numerical experiments, we show that our proposed algorithm is comparable to the existing algorithms in the literature, and in some cases it even returns solutions with superior performance and sparsity pattern.
\end{abstract}

\allowdisplaybreaks

\section{INTRODUCTION}\label{Intro}

The area of distributed control systems has been growing rapidly in the past decade and it has been applied to various real-world problems such as formation control of autonomous vehicles, power networks, transportation networks, mobile wireless networks, only to name a few. In several important applications, the centralized control methodologies cannot be applied due to the lack of access to global information in subsystem level throughout the network.  This design constraint has been motivated researchers to consider the possibility of designing near-optimal sparse feedback controllers for large-scale dynamical networks \cite{lin2012sparse,lin2013design}.

In this paper, we consider a class of $\ell_0$-regularized linear-quadratic (LQ) optimal control problems. This class of problems can be formulated by considering the standard linear-quadratic regulator (LQR) problem and augmenting it with an additional term in the cost objective in order to promote sparsity pattern of the state feedback controller. This class of  $\ell_0$-regularized LQ problems is generally nonconvex and computationally intractable. This is basically due to the presence of $\ell_0$-measure in the cost objective and nonlinear terms in the Lyapunov equation that corresponds to the closed-loop stability of the system. The nonconvex sparsity-promoting term in the cost objective can be convexified by replacing it with $\ell_1$-norm \cite{candes2006near} in order to relax the problem as a $\ell_1$-regularized LQ problem. This is still a nonconvex problem. Our main contributions in this paper are twofold. First, we convexify the $\ell_1$-regularized LQ problem by applying affine approximation techniques to convexify the nonlinear terms in the corresponding Lyapunov equation. Second, we propose an iterative algorithm to solve the $\ell_1$-regularized LQ problem using a series of convexified $\ell_1$-regularized LQ problems. In the literature review, we discuss that our results are close in spirit to \cite{fardad2014design}. In the numerical experiments section, we analyze the performance of our proposed algorithm by means of several simulations and compare our results to the algorithm proposed in \cite{lin2013design}. Our simulation results reveal that our proposed algorithm is comparable to the existing algorithms in the literature, and in some cases it even returns solutions with superior performance and sparsity pattern.



Some works have been done in order to tract problems similar to sparsity-promoting LQR problem. One of the main ideas has been the change of $\ell_0$-norm with convex $\ell_1$-norm and then solving the penalized problem with ADMM method \cite{lin2012sparse,lin2013design}. Another method which has been used to get some bounds on the optimal value of such sparsity-constrained problems works based on projection and gives some useful intuition about optimal value of the corresponding sparsity-constrained problem \cite{martensson2012scalable}. All the methods mentioned so far (except the \cite{fardad2014design}), have been proposed for the continuous-time sparsity-promoting state feedback gain controller. In addition to such continuous-time state feedback gain controllers, some methods have been presented by considering the discrete-time standard LQR problem which some of them propose sparsity-promoting state feedback gain controllers \cite{chizeck1986discrete}, \cite{lavaei2013optimal}, \cite{fardad2014design}, \cite{fazelnia2014convex},\cite{wang2014localized} and \cite{wang2014sparse}. In one of these recent methods, decentralized state feedback gain controller is presented by using some convex relaxations where some graph theoretic proofs are provided to determine the upper bound for rank of such a relaxed SDP solution and if such a rank is equal to 1, then the globally optimal solution can be reconstructed from the relaxed SDP solution \cite{lavaei2013optimal} and \cite{fazelnia2014convex}. However, the solution obtained by such a relaxation-based method is decentralized and is not presented for general sparse controllers. But, it is spanning both finite and infinite horizon discrete-time sparsity-promoting LQR problems \cite{lavaei2013optimal} and \cite{fazelnia2014convex}.


Recently, a method has been revealed which yields sub-optimal structured and sparse feedback gains on the basis of iterative convex programming \cite{fardad2014design}. In such a method, the objective function of optimization problem has an $\mathcal{H}_2$-norm form. Then, the non-convexity of problem is extracted as an inversion of an optimization variable. After that, the equality condition having such a non-convexity is replaced by two complementary inequalities. One of them is characterized by Schur complement and the other one is rewritten as an equality having some penalized term \cite{zhang2006schur}. It should be noted that the non-convex part appeared on such an equality is linearized around an optimal estimate. The achieved results are really notable in comparison with the simplicity of the used iterative convex programming. However, no convergence proof is presented \cite{fardad2014design}. In \cite{wang2014sparse} as well as \cite{fardad2014design}, convex optimization has been used as an effective tool to design a sub-optimal sparse static output feedback controller.

One of the well-known systems studied in Control Theory is spatially decaying systems which has been discussed in recent years \cite{motee2008optimal,motee2013measuring,motee2014sparsity, moteeACC2014}. In these recent works, a large class of spatially decaying systems is classified where their quadratically-optimal feedback controllers inherit spatial decay property from the dynamics of the underlying system. Moreover, they propose a method based on new notions of $q$-Banach algebras where sparsity and spatial localization features of spatially decaying systems can be studied when $q$ is chosen sufficiently small.  In this paper, in first subsection of numerical examples, we employ such a class of systems for verifying our proposed design method. The results are considerably consistent with our expectations. The state feedback gain controllers have patterns very  similar to the structure of the underlying spatially decaying systems with reasonable sparsification.

Another class of spatially distributed systems is cyclic systems which arises in biochemical reactions. Such a class of systems have been investigated mathematically in \cite{tyson1978dynamics,thron1991secant,jovanovic2007remarks,jovanovic2008passivity,siami2015scaling,siami2013robustness}. The second subsection of numerical examples is devoted to considering such well-known systems and their sub-optimal sparse controllers design. Also, we state how the sub-optimal sparsest solution for such a special case can be derived analytically. The interesting fact about such a sub-optimal sparsest solution is that it just has $1$ nonzero term.

This paper is organized as follows: Section \ref{Intro} is an Introduction to the Distributed Optimal Sparse Controllers, our main contributions and diverse methods proposed to design such controllers. Section \ref{ProFor} is devoted to formulate the $\ell_0$-regularized LQR design problem. In Section \ref{sdpres} we see the main part of the paper which shows how $\ell_1$-norm, affine approximation and Schur complement help us to cast the $\ell_1$-regularized LQR problem as a simple convex SDP form which is solved with CVX toolbox \cite{petersen2008matrix}. Section \ref{IterAlg} contains the proposed algorithm which finds the sub-optimal sparse state feedback gain controller. Section \ref{NumExp} gives some numerical experiments of spatially decaying systems and cyclic systems which show how useful the new Affine Approximation-based formulation presented in Section \ref{sdpres} is. Finally, Section \ref{Discon} concludes the paper together with mentioning some future work at the end of it.
\section{Problem Formulation} \label{ProFor}
We consider the class of linear time-invariant  systems  
\begin{equation}
\dot{x}= Ax+Bu~~~~\textrm{with}~~~~x(0)=x_0, 
\end{equation}
where $A \in \mathbb{R}^{n \times n}$ and $B \in \mathbb{R}^{n \times m}$. It is assumed that the pair $(A,B)$ is controllable and the initial condition $x_0$ is drawn from a standard normal distribution (i.e., with zero mean and unit standard deviation). Let us define the following $\ell_0$-regularized Linear-Quadratic (LQ) optimal control problem:\\
\rule{8.65cm}{0.8pt}\\
{\bf $\bf \ell_0$-Regularized LQ  ($\bf \ell_0$-RLQ) Problem:}
\begin{eqnarray*}
& & \hspace{-1.73cm} \underset{x,u,K}{\textrm{Minimize}}~\mathbb{E}\left\{ \int_{0}^{\infty} \big( x^T Q x + u^T R u \big) dt \right\} + \alpha_1 \| K \|_{\ell_0} \\
& & \hspace{-1.7cm} \textrm{subject to:    }\\
& & \hspace{-0.1cm} \dot{x}=Ax+Bu, 
\\
& & \hspace{-0.1cm} u=Kx, \\ 
& & \hspace{-0.1cm} x(0)=x_0, \\ 
& & \hspace{-0.1cm} K ~\textrm{: stabilizing.} 
\end{eqnarray*}
\rule{8.65cm}{0.8pt}\\
In this formulation, $K \in \mathbb{R}^{m \times n}$ is the state feedback gain matrix, $R \succ 0$ is the control weight matrix, $Q \succeq 0$ is the state weight matrix, and $\|K\|_{\ell_0}$ is the number of nonzero elements of matrix $K$ \cite{horn2012matrix}. The $\ell_0$-RLQ problem  can be viewed as a penalized standard LQR problem, where the penalty term $\|K\|_{\ell_0}$ is augmented to improve sparsity pattern of the state feedback gain $K$.

By incorporating constraints of $\ell_0$-RLQ problem, we can cast the objective functional in this problem in a simpler form using the following standard procedure
\begin{eqnarray}
& & \hspace{-.8cm} \int_{0}^{\infty} \big( x^T Q x + u^T R u \big) dt  = \int_{0}^{\infty} \mathbf{Tr} \big( x^T(Q + K^T R K)x \big) dt. \nonumber\\
& & \hspace{-.4cm} = ~\int_{0}^{\infty} \mathbf{Tr} \big( x_0^T e^{(A+BK)^Tt}(Q + K^T R K)e^{(A+BK)t}x_0 \big) dt  \nonumber \\
& & \hspace{-.4cm} = ~\mathbf{Tr}\big(X_0X \big), \label{cost-1}
\end{eqnarray}
where $X_0=x_0x_0^T$ and  
\begin{equation*}
X=\int_0^\infty e^{(A+BK)^Tt}(Q+K^TRK)e^{(A+BK)t}dt.
\end{equation*}
In order to satisfy the last constraint in $\ell_0$-RLQ problem, we apply the standard Lyapunov theorem and conclude that $K$ is stabilizing if and only if $X$ is the unique positive definite solution of  the following Lyapunov equation \cite{dullerud2000course}
\begin{equation*}
(A+BK)^TX+X(A+BK)+Q+K^TRK = 0.
\end{equation*}
In the final step, by taking expectation from the equivalent cost \eqref{cost-1} it follows that $\mathbb{E} \{ \mathbf{Tr}(X_0 X) \}= \mathbf{Tr}(X)$. This can be done as trace is a linear operator and the  expectation operator distributes over it. As a result of our simplifications, $\ell_0$-RLQ problem can be rewritten in the following compact form:
\rule{8.65cm}{0.8pt}\\
{\bf  Equivalent Form of the $\bf \ell_0$-RLQ Problem:}
\begin{eqnarray}
& & \hspace{-1cm} \underset{X,K}{\textrm{Minimize}} ~\mathbf{Tr}(X) +\alpha_1 \|K\|_{\ell_0} \label{obj-1}\\
& & \hspace{-1cm} \textrm{subject to:  } \nonumber \\
& & \hspace{-1cm} (A+BK)^TX+X(A+BK)+Q+K^TRK = 0, \label{const-1}\\
& & \hspace{-1cm} X \succ 0 \label{const-2}. 
\end{eqnarray}
\rule{8.65cm}{0.8pt}\\
The equivalent form of the $\ell_0$-RLQ problem is in general an NP-hard problem and computationally intractable. This is due to the $\ell_0$-sparsity measure in the cost function \eqref{obj-1}. The $\ell_0$-RLQ problem can be relaxed by replacing  the $\ell_0$-sparsity measure with its best convex approximation, i.e.,  $\ell_1$-norm \cite{candes2006near} in order to get the following relaxed optimal control problem:
\\\rule{8.65cm}{0.8pt}\\
{\bf $\bf \ell_1$-Regularized LQ  ($\bf \ell_1$-RLQ) Problem:}
\begin{eqnarray}
& & \hspace{-1cm} \underset{X,K}{\textrm{Minimize}} ~\mathbf{Tr}(X) +\alpha_1 \|K\|_{\ell_1} \label{obj-2}\\
& & \hspace{-1cm} \textrm{subject to:  } \nonumber \\
& & \hspace{-1cm} (A+BK)^TX+X(A+BK)+Q+K^TRK = 0, \label{const-3}\\
& & \hspace{-1cm} X \succ 0 \label{const-4}. 
\end{eqnarray}
\rule{8.65cm}{0.8pt}\\
The $\ell_1$-RLQ problem is still nonconvex due to nonlinear constraint \eqref{const-3}. This constraint has differentiable algebraic forms. In the following section, we employ an affine approximation method to convexify the $\ell_1$-RLQ problem. Then, we use our results to propose an iterative algorithm to compute sub-optimal solutions for the original $\ell_0$-RLQ problem. 

\begin{remark}
Alternative forms of $\ell_1$-RLQ problem have been considered before in \cite{lin2013design}, where the authors employ an ADMM-based approach to handle nonlinear constraint \eqref{const-3}. In our numerical experiments in Section \ref{NumExp}, we will repeatedly compare the performance of our proposed algorithm to that of  \cite{lin2013design}. 
\end{remark}

\section{SDP-Restriction of the RE$\bf \ell_1$-RLQ Problem}\label{sdpres}

In this section, we derive a SDP-restriction of a relaxed form of
$\ell_1$-RLQ problem. In the first step using the following lemma, we rewrite $\ell_1$-RLQ problem in a new equivalent form where the source of nonlinearity in $\ell_1$-RLQ problem is extracted and isolated as a new constraint, i.e., inequality \eqref{nonlinear-1}. Then, we will show how to utilize affine approximations to estimate that nonlinear constraint and come up with a SDP-restriction of a relaxed form of $\ell_1$-RLQ problem.  Let us denote $\mathbb{S}^{n}_{++}$ and $\mathbb{S}^{n}_{+}$ to be the positive definite cone and positive semi-definite cone of $n \times n$ matrices, respectively.

\begin{lemma} \label{prop1}
The $\ell_1$-RLQ problem is equivalent to the following optimization problem:\\
\rule{8.65cm}{0.8pt}
{\bf Equivalent Form of the $\bf \ell_1$-RLQ Problem (E$\bf \ell_1$-RLQ):}
\begin{eqnarray}
& & \hspace{-1.2cm} \underset{X,K,F,P,Y}{\textrm{Minimize}} ~ \mathbf{Tr}(X)+\alpha_1 \|K\|_{\ell_1} \nonumber \\
& & \hspace{-1.2cm} \textrm{subject to:  } \nonumber \\
& & \hspace{-0.2cm} \begin{bmatrix} -Q-F+Y & 2X-P^T & P^T\\ 2X-P & I & 0\\P & 0 & \frac{1}{\delta} I\end{bmatrix} \succeq 0, \nonumber\\
& & \hspace{-0.2cm} P = X-\frac{A+BK}{2}, \nonumber\\
& & \hspace{-0.2cm} X \succeq \epsilon_1I, \nonumber\\
& & \hspace{-0.2cm} \begin{bmatrix} Y & P^T \\P & \frac{1}{1+\delta}I\end{bmatrix} \succeq 0, \nonumber\\
& & \hspace{-0.2cm} \begin{bmatrix} F & K^T\\ K & R^{-1}\end{bmatrix} \succeq 0, \nonumber\\
& & \hspace{-0.2cm} Y \preceq (1+\delta)P^TP, \label{nonlinear-1}
\end{eqnarray}
\rule{8.65cm}{0.8pt}
where $X \in \mathbb{S}^{n}_{++}$, $Y \in \mathbb{S}^{n}_{+}$, $F \in \mathbb{S}^{n}_{+}$, $K \in \mathbb{R}^{m \times n}$ and $P \in \mathbb{R}^{n \times n}$. 
\end{lemma}
\begin{proof}
See Appendix.
\end{proof}

In the following proposition, we take an important step toward dealing with the nonconvex term \eqref{nonlinear-1} in E$\ell_1$-RLQ problem.
\begin{proposition}
The E$\ell_1$-RLQ problem can be relaxed as:
\rule{8.65cm}{0.8pt}
{\bf Relaxed  E$\bf \ell_1$-RLQ Problem (RE$\bf \ell_1$-RLQ):} 
\begin{eqnarray}
& & \hspace{-1.2cm} \underset{X,K,F,P,Y}{\textrm{Minimize}}~ \mathbf{Tr}(X)+\alpha_1 \|K\|_{\ell_1} \nonumber \\
& & \hspace{-1.2cm} \textrm{subject to:  } \nonumber \\
& & \hspace{-0.2cm} \begin{bmatrix} -Q-F+Y & 2X-P^T & P^T\\ 2X-P & I & 0\\P & 0 & \frac{1}{\delta} I\end{bmatrix} \succeq 0, \nonumber \\
& & \hspace{-0.2cm} P = X-\frac{A+BK}{2}, \nonumber \\
& & \hspace{-0.2cm} X \succeq \epsilon_1I, \nonumber \\
& & \hspace{-0.2cm} \begin{bmatrix} Y & P^T\\P & \frac{1}{1+\delta}I\end{bmatrix} \succeq 0, \nonumber \\
& & \hspace{-0.2cm} \begin{bmatrix} F & K^T\\ K & R^{-1}\end{bmatrix} \succeq 0, \nonumber \\
& & \hspace{-0.2cm} \|Y-(1+\delta)P^TP\|_{*} \le \epsilon, \label{ineq-1}
\end{eqnarray}
\rule{8.65cm}{0.8pt}
where $\|.\|_{*}$ denotes the nuclear norm of $(.)$ which is defined as sum of singular values of $(.)$. And also for a given matrix $U$, we have $\|U\|_{*}=\mathbf{Tr}(\sqrt{U^TU})$ where $\sqrt{U^TU}$ is a matrix for which we have $U^TU=(\sqrt{U^TU})^2$.
\end{proposition}
\begin{proof}
Our goal is to prove that RE$\ell_1$-RLQ problem is a relaxation of E$\ell_1$-RLQ problem. It suffices to show that the feasible set of RE$\ell_1$-RLQ problem, which is represented by $\mathcal{F}_1$, contains the feasible set of E$\ell_1$-RLQ problem, which is represented  by $\mathcal{F}_2$. Moreover, suppose that $\mathcal{F}_3$ denotes the set of all feasible points specified by all constraints  of RE$\ell_1$-RLQ problem except constraint \eqref{ineq-1} and $\mathcal{F}_4$ denotes the set of all feasible points determined only by constraint \eqref{ineq-1}. It is straightforward to verify that $\mathcal{F}_2 \subseteq \mathcal{F}_3$. This is because the set of constraints specifying $\mathcal{F}_3$ is a subset of set of constraints specifying $\mathcal{F}_2$. This implies that for a feasible point $(X_1,K_1,F_1,P_1,Y_1) \in \mathcal{F}_2$, it follows that  $(X_1,K_1,F_1,P_1,Y_1) \in \mathcal{F}_3$. 

Since $(X_1,K_1,F_1,P_1,Y_1) \in \mathcal{F}_2$, $Y_1=(1+\delta)P_1^TP_1$ is satisfied. Therefore, $\|Y_1-(1+\delta)P_1^TP_1\|_{*}=0$ is also satisfied and consequently $\|Y_1-(1+\delta)P_1^TP_1\|_{*}=0 \le \epsilon$ is also true. Thus, $(X_1,K_1,F_1,P_1,Y_1) \in \mathcal{F}_4$. Considering the fact that $\mathcal{F}_3 \cap \mathcal{F}_4 = \mathcal{F}_1$ is satisfied, it implies that $(X_1,K_1,F_1,P_1,Y_1) \in \mathcal{F}_1$. Thus, it proves our claim that $\mathcal{F}_1 \supseteq \mathcal{F}_2$.
\end{proof}
\begin{remark}
It should be noted that there is no much difference between choice of $\|.\|_*$ and $\|.\|$ in \eqref{ineq-1}. It just affects the choice of parameter $\epsilon$.
\end{remark}
In the following step, we employ a real analysis idea \cite{royden1988real} to convexify the right hand side of  inequality \eqref{ineq-1} in RE$\ell_1$-RLQ problem, which is the only nonconvex term appearing in RE$\ell_1$-RLQ problem. By performing some matrix calculus, we calculate the convex affine approximation of $M=(1+\delta) P^T P$  as follows
\begin{equation*}
N =(1+\delta) \big(P^T\bar{P} + \bar{P}^T(P-\bar{P})\big),
\end{equation*}
where $\bar{P}$ is an estimate of $P$. We refer to  \cite{petersen2008matrix} for more details on affine approximations of functions of matrices.

So, we substitute the nonconvex condition $\|Y-M\|_{*} \le \epsilon$ with convex condition $\|Y-N\| \le \frac{\epsilon}{n}$ where $\|.\|$ denotes operator norm which is defined as largest singular value of $(.)$. Finally, using the proof of proposition $2$.$1$. in \cite{recht2010guaranteed}, we get the following proposition:
\begin{proposition} \label{propresrel}
The RE$\ell_1$-RLQ problem can be restricted as the following SDP form:
\rule{8.65cm}{0.8pt}
{\bf SDP Restriction of the RE$\bf \ell_1$-RLQ Problem (SDP-R-RE$\bf \ell_1$-RLQ):}
\begin{eqnarray}
& & \hspace{-1.2cm} \underset{X,K,F,P,Y}{\textrm{Minimize}}~ \mathbf{Tr}(X)+\alpha_1 \|K\|_{\ell_1} \nonumber \\
& & \hspace{-1.2cm} \textrm{subject to:  } \nonumber \\
& & \hspace{-0.2cm} \begin{bmatrix} -Q-F+Y & 2X-P^T & P^T\\ 2X-P & I & 0\\P & 0 & \frac{1}{\delta} I\end{bmatrix} \succeq 0, \nonumber \\
& & \hspace{-0.2cm} P = X-\frac{A+BK}{2}, \nonumber\\
& & \hspace{-0.2cm} X \succeq \epsilon_1 I, \nonumber \\
& & \hspace{-0.2cm} \begin{bmatrix} Y & P^T \\ P & \frac{1}{\delta+1}I\end{bmatrix} \succeq 0, \nonumber \\
& & \hspace{-0.2cm} \begin{bmatrix} F & K^T\\ K & R^{-1}\end{bmatrix} \succeq 0, \nonumber \\
& & \hspace{-0.2cm} \begin{bmatrix} \frac{\epsilon}{n} I & Y-N\\ Y-N& \frac{\epsilon}{n} I\end{bmatrix} \succeq 0. \nonumber
\end{eqnarray}
\rule{8.65cm}{0.8pt}
\end{proposition}
\begin{proof}
Let us suppose that $\mathcal{F}_5$ denotes the set specified by constraint $\|Y-N\| \le \frac{\epsilon}{n}$ and $\mathcal{F}_6=\mathcal{F}_3 \cap \mathcal{F}_5$.\\
So, we should prove that $\mathcal{F}_6 \subseteq \mathcal{F}_1$. Suppose that $(X_2,K_2,F_2,P_2,Y_2)$ be a feasible point in $\mathcal{F}_6$. Hence, we have $(X_2,K_2,F_2,P_2,Y_2) \in \mathcal{F}_3$ and $(X_2,K_2,F_2,P_2,Y_2) \in \mathcal{F}_5$. Using the fact indicated in inequality (2.1) in \cite{recht2010guaranteed}, we have $\|Y-N\|_{*} \le n\|Y-N\| \le \epsilon$. Also, it is easy to check that $M-N=(\delta+1)(P-\bar{P})^T(P-\bar{P})$ is satisfied. Thus, together with $Y \succeq M$, it implies that $Y \succeq M \succeq N$. Now, $Y-M \preceq Y-N$ and $(Y-N)-(Y-M)=M-N \succeq 0$ implies that $\mathbf{Tr}(Y-N)-\mathbf{Tr}(Y-M) \ge 0$. Since $Y-M$ and $Y-N$ are both symmetric matrices, their traces are same as their nuclear norms. Hence, we will have $\|Y-M\|_{*} \le \|Y-N\|_{*}$. Thus, $\|Y-M\|_{*} \le \epsilon$ and subsequently $(X_2,K_2,F_2,P_2,Y_2) \in \mathcal{F}_4$ are resulted. Consequently, $(X_2,K_2,F_2,P_2,Y_2) \in \mathcal{F}_4 \cap \mathcal{F}_3=\mathcal{F}_1$. Subsequently, $\mathcal{F}_6 \subseteq \mathcal{F}_1$ is satisfied and proof is done.
\end{proof}

\begin{remark}
In the following, we try to explain the existing technical issue with the method proposed by \cite{fardad2014design}. In problem (RLX) in \cite{fardad2014design}, very briefly, the authors have considered the $Y-N$ as difference of two positive semi-definite matrices $Z_{+}$ and $Z_{-}$ and forced trace of $Z_{+}$ to zero by minimizing some multiplier of it in objective function, then they have made a conclusion that $Y-N$ becomes negative semi-definite and this together with $Y \succeq X^{-1}$ forces $Y=X^{-1}$. We know that $Y \succeq N$,thus $Y-N$ cannot be negative semi-definite. We are just able to minimize $\|Y-N\|$. The point is that when we force trace of $Z_{+}$ to zero, it cannot be necessarily implied that $Z_{+}-Z_{-}$ tends to $-Z_{-}$ and so it is negative semi-definite. Because, $Z_{-}$ is an optimization variable and its trace can be either tend to zero and if the rate of tending be greater than of $Z_{+}$ one, $Z_{+}-Z_{-}$ cannot be negative semi-definite. Also, if we run algorithm proposed by \cite{fardad2014design}, we will see that not only $Z_{+}-Z_{-}$ is not negative semi-definite, but also it is positive semi-definite which is previously theoretically guaranteed.Thus, we minimize $\|Y-N\|$ instead of using such a method presented in \cite{fardad2014design}.
\end{remark}

\begin{remark}
The other point is that we have characterized the approximated constraint via some parameter called $\epsilon$ in feasible set. However, authors in \cite{fardad2014design}, has done such a goal by considering some parameter called $\alpha_2$ in objective function. Although, these two approaches are equivalent by ignoring the issue arisen in $Z_{+}-Z_{-}$, our proposed method gives much better insight. Because, in proposition \ref{propresrel}, we show how $\|Y-N\|$ and $\|Y-M\|$ are related to each other.
\end{remark}

We consider the SDP-R-RE$\ell_1$-RLQ problem and solve it iteratively by using CVX toolbox. The procedure is as follows:

Firstly, we start to solve the SDP-R-RE$\ell_1$-RLQ problem for initial guess $P^0$. Then, at each step we update the previous values of $\bar{P}$ by its current computed value and continue this process until we reach to the desired level of accuracy. It should be mentioned that the parameter $\epsilon$ is firstly chosen as a sufficiently large number to find a primitive solution and then it decreases in a geometric progression form in order to achieve the sub-optimal solution. In section \ref{IterAlg}, we propose and discuss such an algorithm in detail.
\section{An Iterative Algorithm to Solve for Sub-optimal Sparse Solutions}\label{IterAlg}
In order to compute a near-optimal sparse feedback controller, we propose an iterative algorithm which works based on solving a sequence of SDP forms \cite{boyd2009convex}. The abstract notion of our proposed algorithm is as follows:\\ 
\rule{8.65cm}{0.8pt}\\
\begin{enumerate}
\item Given $\alpha >0$, $\beta < 1$, $\epsilon_2 >0$, $\bar{P}:=P^0$.
\item For i=1,2,... do
\item Set $\epsilon:= n \alpha \beta^{i-1}$
\item Solve SDP-R-RE$\ell_1$-RLQ problem to find optimal $P^*$.
\item If $\frac{\|P^*-\bar{P}\|_{F}}{\|P^*\|_{F}} \ge \epsilon_2$ or $\frac{\|Y-(\delta+1){{P^*}^T P^*\|_F}}{\|Y\|_F} \ge \epsilon_2$ then, set $\bar{P}$ equal to optimal $P^*$, else, set $P^{opt}=P^*$ break \label{item5}
\item End for.
\end{enumerate}
\rule{8.65cm}{0.8pt}
where $\|.\|_{F}$ denotes the Frobenius norm of $(.)$ which is defined as $\|.\|_F=\sqrt{\mathbf{Tr}((.)^T(.))}$.

Our selection of Frobenius norm doesn't provide any major advantage and it can be replaced by operator norm or nuclear norm \cite{horn2012matrix}. It should be considered that $P^0$ is initiated as $X^0-\frac{A+BK^0}{2}$ where $(K^0)$ is the solution of standard LQR problem and $X^0$ is equal to $\int_0^\infty e^{(A+BK^0)^Tt}(Q+{K^0}^TRK^0)e^{(A+BK^0)t}dt$ which can be computed via solving a Lyapunov equation. The inequality appeared on item \ref{item5} of the proposed algorithm are considered in a normalized form to reach higher accuracy.

As we will see in section \ref{NumExp}, our numerous experiments show that our proposed algorithm converges and suggests a sub-optimal sparse solution for $\ell_1$-RLQ problem, however, we have no proof on its convergence.

\section{Numerical Experiments}\label{NumExp}
\subsection{Spatially Decaying Systems}
Here we give an example of spatially decaying systems which have been considered in \cite{bamieh2002distributed,motee2008optimal,motee2013measuring,motee2014sparsity}. In such systems the absolute value of entries of $A$-matrix decreases spatially, i.e. when we start from an entry on main diagonal and get away from it, the absolute value of entries falls off.

As a sub-class of such spatially decaying systems, let us consider the sub-exponentially decaying systems. The ${ij}^{th}$ element of a random sub-exponentially decaying system $A$ is defined by $A_{ij}=c_ie^{-\alpha_A |i-j|^{\beta_A}}$ where $c_i$ is the real random number corresponded to the $i^{th}$ row and generated by $randn$ command of MATLAB, $\alpha_A$ is a positive real number and $\beta_A$ is a real number between 0 and 1 \cite{guide1998mathworks}. Note that $\alpha_A$ determines the width of the band of the matrix $A$ and $\beta_A$ specifies the rate of decaying in sub-exponentially decaying system modeled by $A$. For large values of $\alpha_A$ we get narrow band width and for large values of $\beta$ we get high rate of spatially decaying. Since we have $A_{ii}=c_i$, $c_i$ appears on $i^{th}$ entry of main diagonal of matrix $A$. By taking $\alpha_A= 5$, $\beta_A=0.5077$ and some randomly generated $c_{i}$'s we get the following sub-exponentially decaying $A$-matrix:\\
\[ A = \footnotesize {\begin{bmatrix} $0.5377$  & $0.0036$ &  $0.0004$  & $0.0001$ &  $0.0000$  & $0.0000$\\
 $0.0124$ &  $1.8339$  & $0.0124$  & $0.0015$  & $0.0003$ &  $0.0001$\\
    $-0.0018$  &  $-0.0152$ &   $-2.2588$  &  $-0.0152$  &  $-0.0018$ &   $-0.0004$\\
    $0.0001$ &  $0.0007$  & $0.0058$ &  $0.8622$ &  $0.0058$ &  $0.0007$\\
    $0.0000$ &   $0.0001$ &   $0.0003$  &   $0.0021$  &   $0.3188$ &   $0.0021$\\
   $-0.0000$ &  $-0.0001$ &  $-0.0002$ &  $-0.0011$  & $-0.0088$ &  $-1.3077$ \end{bmatrix}}\]
The set of eigenvalues of $A$ is $\{-2.2588, 1.8339, -1.3077, 0.5376, 0.8622, 0.3187\}$. Thus, $A$ is an unstable system.
Assuming the $\alpha_1=0.005$, $\alpha = 0.00001$, $\beta = 0.99$, $\delta=0.001$, $\epsilon_1=0.000001$, $\epsilon_2=0.00005$ and $B=Q=R=I$, Algorithm in section IV proposes the following sub-optimal sparse state feedback gain controller:\\
    \[ K^{opt} = \footnotesize \begin{bmatrix} $-1.6675$ & $-0.0066$ &  $0.0000$ &  $0.0000$ &   $0.0000$ &  $0.0000$\\
   $-0.0113$ &  $-3.9123$  & $0.0000$  & $0.0000$ &  $0.0000$ &  $0.0000$\\
   $0.0000$  &  $0.0000$  & $-0.1992$ &  $0.0000$  &  $0.0000$ &  $0.0000$ \\
   $0.0000$  &  $0.0000$ &  $0.0000$ &  $-2.1760$  &  $-0.0014$ &  $0.0000$\\
    $0.0000$  & $0.0000$  & $0.0000$ &  $-0.0001$ &  $-1.3632$ &   $0.0000$\\
   $0.0000$  &  $0.0000$  & $0.0000$  & $0.0000$ &  $0.0000$ &  $-0.3304$ \end{bmatrix}\]
The set of eigenvalues of $A+BK^{opt}$ consists of $-2.4575$, $-2.0789$, $-1.1298$, $-1.0444$, $-1.3139$ and $-1.6380$. Thus, the proposed sub-optimal sparse feedback gain controller is stabilizing. Also, the optimal value is $9.6968$. The optimal value of the standard LQR problem is $9.6965$ in this case.

In a similar way, we consider a $16 \times 16$ sub-exponentially decaying system and plot its corresponding matrices $A$ and $K$ in a logarithmic gray-scale spectrum. The corresponding parameters are as follows:\\ $\alpha_1=0.005$, $\alpha = 0.00001$, $\beta = 0.99$, $\alpha_A= 5$, $\beta_A=0.8903$, $\delta=0.001$, $\epsilon_1=0.000001$, $\epsilon_2=0.00005$ and $B=Q=R=I$. The logarithmic gray-scale spectrum of matrices $A$ and $K^{opt}$ has been depicted in Fig. \ref{Fig. 1} and Fig. \ref{Fig. 2} respectively. \\

\begin{figure}[ht]
    \centering
    \includegraphics[scale=0.5]{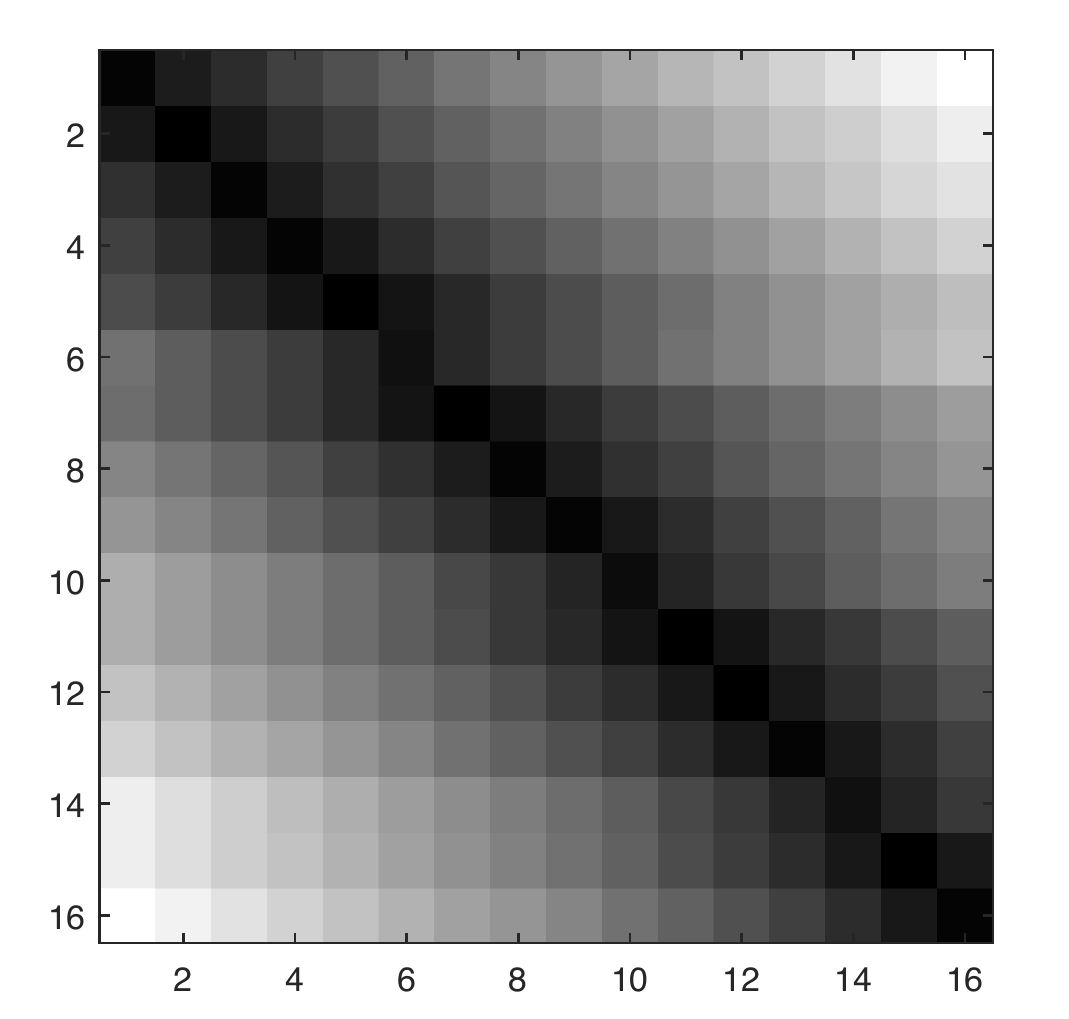}
    \caption{Logarithmic gray-scale spectrum of matrix $A$}
    \label{Fig. 1}
\end{figure}

\begin{figure}[ht]
    \centering
    \includegraphics[scale=0.5]{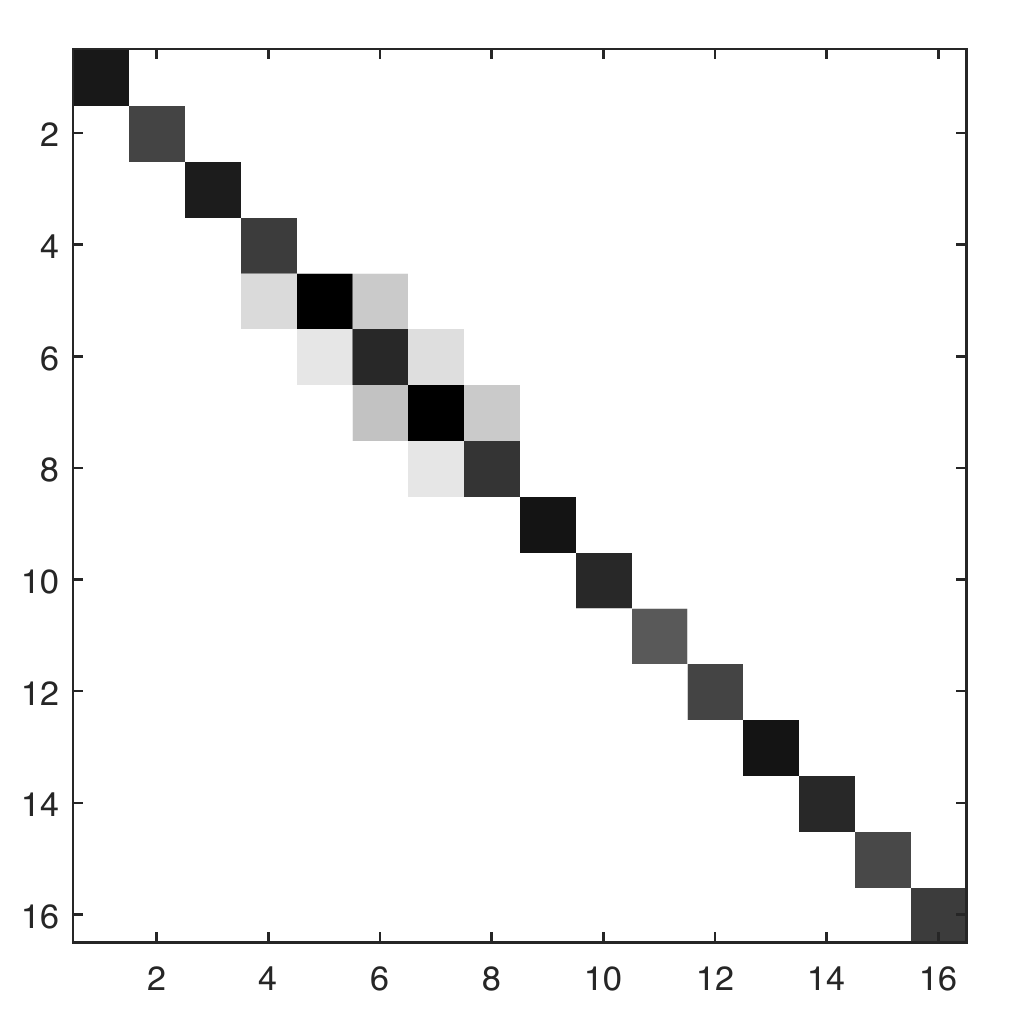} 
    \caption{Logarithmic gray-scale spectrum of matrix $K^{opt}$}
    \label{Fig. 2}
\end{figure}
As it is depicted in Fig. \ref{Fig. 2}, the sub-optimal sparse state feedback gain controller $K^{opt}$ has a sub-exponentially-like structure. In \cite{motee2008decentralized}, it has been shown that if matrices $A$, $B$, $Q$ and $R$ belong to an operator algebra structure, the corresponding standard LQR problem solution inherits such a structure. So, here we see some similar intuitive observation which illustrates such a fact.
\subsection{Cyclic Systems}
Here, we consider cyclic systems to investigate our proposed algorithm in section IV. Cyclic systems arise in various applications such as biochemical reactions which have been widely studied from mathematical point of view in different works such as \cite{tyson1978dynamics,thron1991secant,jovanovic2007remarks,jovanovic2008passivity,siami2015scaling,siami2013robustness}. The $A$-matrix of a cyclic system is specified satisfying the following properties:
\begin{itemize}
\item $A_{ii}$ is negative for all $i \in \{1,...,n\}$.
\item $A_{i (i-1)}$ is positive for all $i \in \{2,...,n\}$.
\item $A_{1 n}$ is negative.
\end{itemize}
Let us assume that $A$ is a $10 \times 10$ random cyclic matrix generated by MATLAB. Then, the sparsity pattern of sparse state feedback gain controller for such a system with following parameters: $\alpha_1=0.5$, $\alpha = 0.00001$, $\beta = 0.99$, $\delta=0.001$, $\epsilon_1=0.000001$, $\epsilon_2=0.00005$, $B=Q=R=I$ has been depicted in Fig. \ref{Fig. 3}.
\begin{figure}[ht]
    \centering
    \includegraphics[scale=0.5]{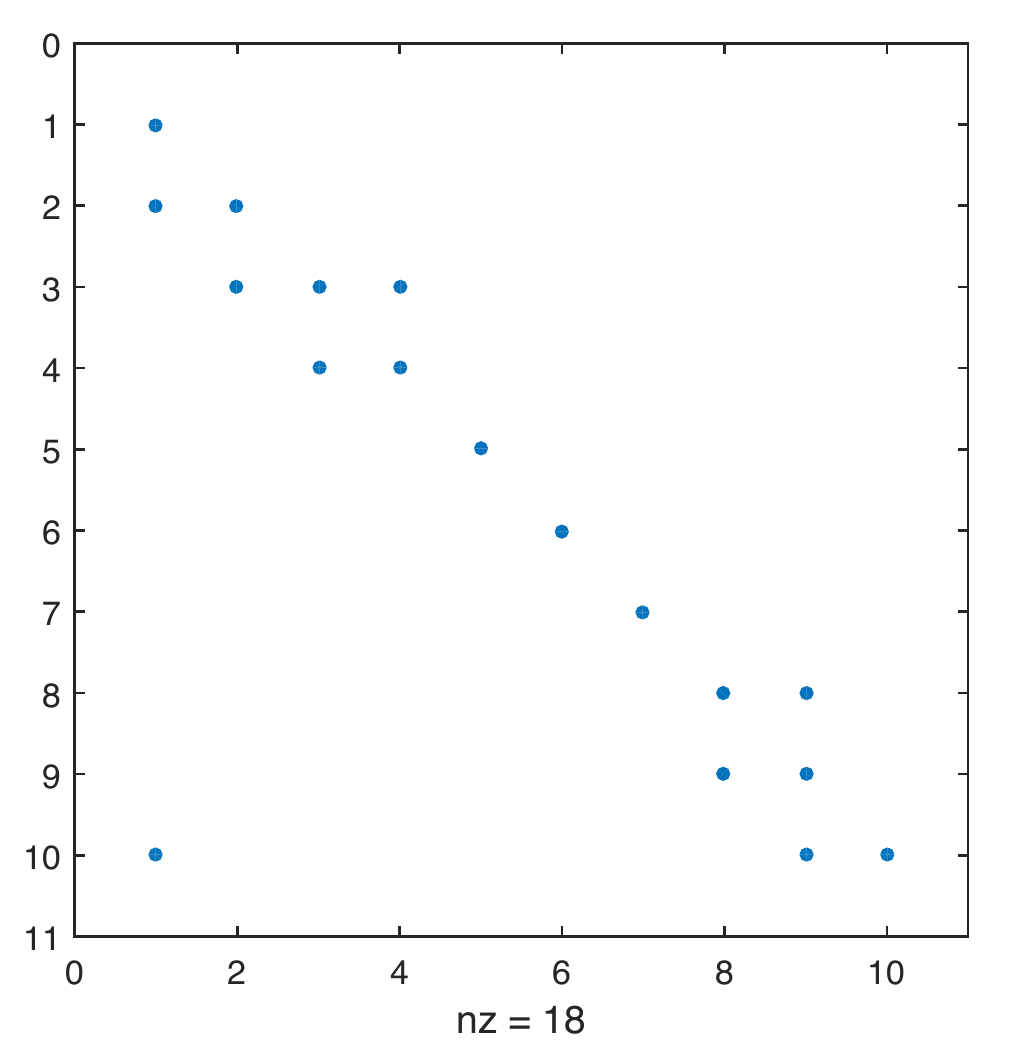}
    \caption{Sparsity pattern of matrix $K^{opt}$ for $\alpha_1=0.5$}
    \label{Fig. 3}
\end{figure}
As it is pointed in Fig. \ref{Fig. 3} the number of nonzero terms of the proposed sub-optimal state feedback gain controller is equal to $18$ i.e. $82 \%$ of entries of the proposed sub-optimal state feedback gain controller $K^{opt}$ is equal to $0$. Thus, it is sparse. Also, the optimal value of the standard LQR  and optimal value of our proposed sub-optimal state feedback gain controller are $16.3760$ and $17.8602$ respectively. So, for $\alpha_1=0.5$, if the quadratic cost gets increased by $9.06 \%$, the sparsity is improved by $82 \%$.

Table \ref{Tab1} demonstrates the effect of penalizing parameter $\alpha_1$ on optimal value $J^{opt}$ and cardinality of sub-optimal controller $\|K^{opt}\|_{\ell_0}$ in our proposed method (i.e. CVX) and the method proposed by \cite{lin2013design} (i.e. ADMM).
\begin{table}[ht!]
\centering
\caption{quadratic cost and cardinality vs penalizing parameter for our proposed method and the method proposed by \cite{lin2013design} for a $10 \times 10$ random cyclic matrix.}
\label{Tab1}
\begin{tabular}{ |p{1.2cm}|p{1cm}|p{1.2cm}|p{1cm}|p{1.2cm}|  } 
\hline
\multicolumn{5}{|c|}{$J^{LQR}=16.3760$ and $\|K^{LQR}\|_{\ell_0}=100$}\\
\hline
log scale & CVX & CVX & ADMM & ADMM\\
\hline
$\alpha_1$ & $J^{opt}$ & $\|K^{opt}\|_{\ell_0}$ & $J^{opt}$ & $\|K^{opt}\|_{\ell_0}$\\
\hline
$0.001$ & $16.3761$ & $61$ & $16.3760$ & $94$\\
$0.002$ & $16.3762$ & $56$ & $16.3760$ & $64$\\
$0.003$ & $16.3765$ & $55$ & $16.3761$ & $58$\\
$0.004$ & $16.3768$ & $55$ & $16.3763$ & $57$\\
$0.005$ & $16.3773$ & $54$ & $16.3765$ & $55$\\
$0.006$ & $16.3778$ & $52$ & $16.3768$ & $54$\\
$0.007$ & $16.3784$ & $48$ & $16.3772$ & $51$\\
$0.008$ & $16.3790$ & $47$ & $16.3777$ & $48$\\
$0.009$ & $16.3797$ & $47$ & $16.3782$ & $48$\\
$0.01$ & $16.3805$ & $45$ & $16.3789$ & $46$\\
$0.02$ & $16.3914$ & $40$ & $16.3875$ & $40$\\
$0.03$ & $16.4070$ & $38$ & $16.4020$ & $39$\\
$0.04$ & $16.4260$ & $37$ & $16.4198$ & $38$\\
$0.05$ & $16.4451$ & $35$ & $16.4389$ & $35$\\
$0.06$ & $16.4663$ & $32$ & $16.4590$ & $32$\\
$0.07$ & $16.4896$ & $31$ & $16.4812$ & $31$\\
$0.08$ & $16.5152$ & $30$ & $16.5057$ & $30$\\
$0.09$ & $16.5422$ & $28$ & $16.5318$ & $28$\\
$0.1$ & $16.5677$ & $27$ & $16.5597$ & $28$\\
$0.2$ & $16.8308$ & $22$ & $16.7944$ & $22$\\
$0.3$ & $17.1527$ & $20$ & $17.1126$ & $21$\\
$0.4$ & $17.4952$ & $19$ & $17.4584$ & $19$\\
$0.5$ & $17.8602$ & $18$ & $17.8118$ & $18$\\
$0.6$ & $18.2200$ & $15$ & $18.1611$ & $15$\\
$0.7$ & $18.5817$ & $15$ & $18.5123$ & $15$\\
$0.8$ & $18.9483$ & $15$ & $18.8688$ & $15$\\
$0.9$ & $19.3169$ & $15$ & $19.2274$ & $15$\\
$1$ & $19.6853$ & $14$ & $19.5860$ & $14$\\
$2$ & $23.1799$ & $13$ & $22.9991$ & $13$\\
$3$ & $25.5533$ & $13$ & $26.0798$ & $13$\\
$4$ & $26.9975$ & $12$ & $28.8795$ & $12$\\
$5$ & $28.6990$ & $11$ & $31.3772$ & $10$\\
$6$ & $30.1690$ & $11$ & $33.6764$ & $10$\\
$7$ & $31.5150$ & $10$ & $35.8389$ & $10$\\
$8$ & $32.9970$ & $10$ & $37.8935$ & $10$\\
$9$ & $34.4342$ & $10$ & $39.8488$ & $10$\\
$10$ & $35.2497$ & $10$ & $41.7020$ & $10$\\
\hline
\end{tabular}
\\
\end{table}
Based on data set shown in Table \ref{Tab1}, Fig. \ref{Fig. 4} depicts how penalizing parameter, sparsity level and performance degradation relate to each other. As it is seen intuitively, Fig. \ref{Fig. 4} confirms this fact that as penalizing parameter $\alpha_1$ increases, the sparsity level decreases, however, performance degradation increases. Thus, there is a trade-off between sparsity level and performance degradation. As an instance for $\alpha_1=0.1$, by increasing $1.17 \%$ and $1.12 \%$ in quadratic cost for CVX and ADMM, respectively, $73 \%$ and $72 \%$ of sub-optimal sparse state feedback gain controller entries get equal to $0$, respectively.

\begin{figure}[ht]
    \centering
    \includegraphics[scale=0.55]{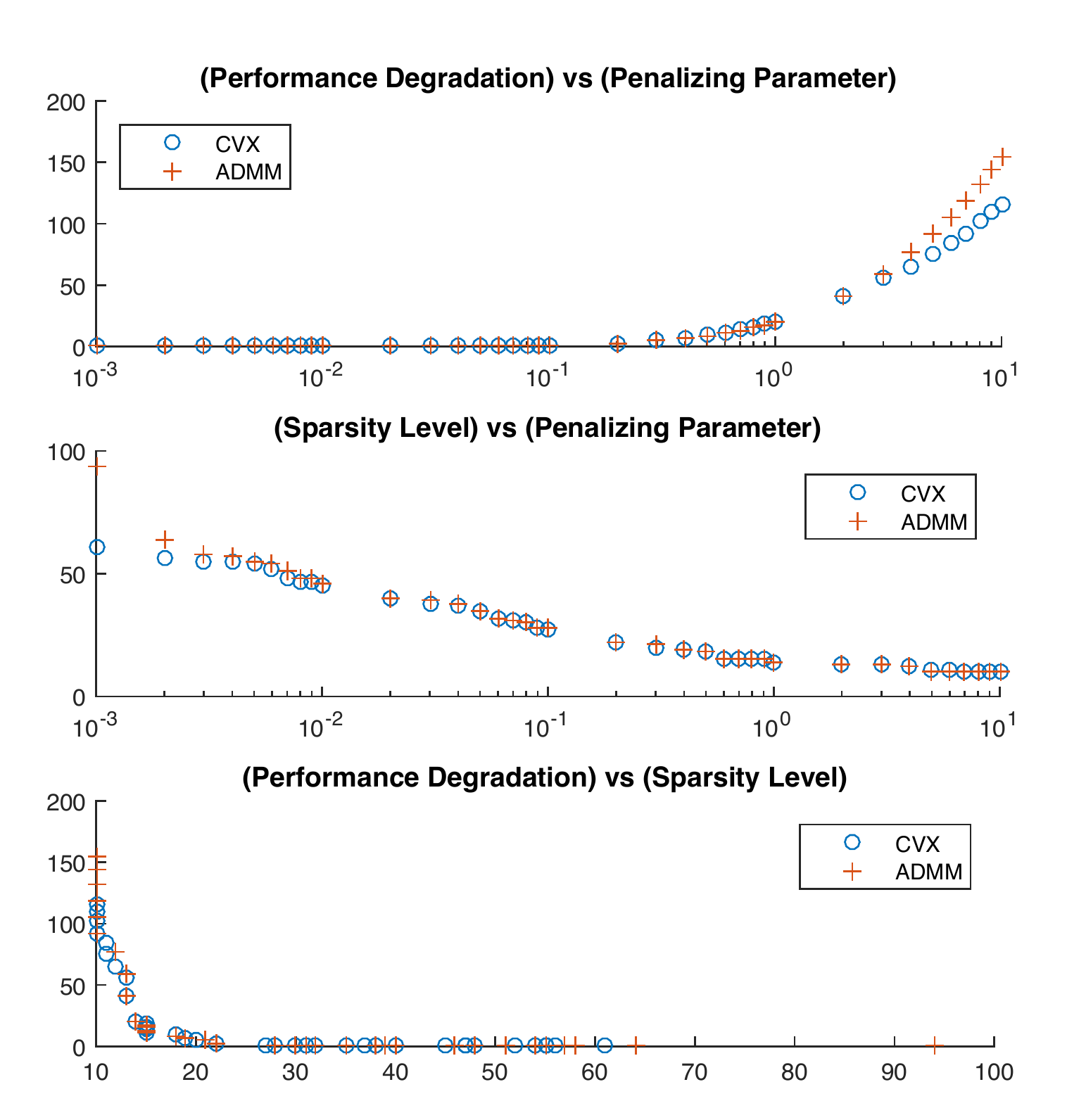} 
    \caption{The Performance Degradation $\frac{J^{opt}-J^{LQR}}{J^{LQR}} \times 100 ~ \%$ vs Penalizing Parameter $\alpha_1$, The Sparsity Level $\frac{\|K^{opt}\|_{\ell_0}}{\|K^{LQR}\|_{\ell_0}} \times 100 ~ \%$ vs Penalizing Parameter $\alpha_1$ and The Performance Degradation $\frac{J^{opt}-J^{LQR}}{J^{LQR}} \times 100 ~ \%$ vs Sparsity Level $\frac{\|K^{opt}\|_{\ell_0}}{\|K^{LQR}\|_{\ell_0}} \times 100 ~ \%$ corresponded to data set shown in Table \ref{Tab1} for both CVX and ADMM methods.}
    \label{Fig. 4}
\end{figure}
According to the secant condition stated in \cite{thron1991secant} for the stability of the linear cyclic systems, we have the following sufficient condition:
\begin{equation*}
\frac{(\prod_{i=2}^{n}(A^c_{i(i-1)}))(-A^c_{1n})}{\prod_{i=1}^{n}(-A^c_{ii})} < (\sec(\frac{\pi}{n}))^n,
\end{equation*}
where by superscript $^c$, we mean the closed loop.\\
By doing some basic math, it yields that we have:
\begin{enumerate}
\item For $K_{jj}$ where $j \in \{1,...,n\}$, we get the following condition:\\
\begin{equation*}
K_{jj} < -\frac{(\cos(\frac{\pi}{n}))^n(\prod_{i=2}^{n}(A_{i(i-1)}))(-A_{1n})}{\prod_{i=1,i \neq j}^{n}(-A_{ii})}-A_{jj}.
\end{equation*}
\item For $K_{j(j-1)}$ where $j \in \{2,...,n\}$, we get the following condition:
\begin{equation*}
-A_{j(j-1)} < K_{j(j-1)},
\end{equation*}
\begin{equation*}
K_{j(j-1)} < \frac{(\sec(\frac{\pi}{n}))^n\prod_{i=1}^{n}(-A_{ii})}{(\prod_{i=2, i \neq j}^{n}(A_{i(i-1)}))(-A_{1n})}-A_{j(j-1)}.
\end{equation*}
\item For $K_{1n}$, we get the following condition:
\begin{equation*}
-A_{1n} > K_{1n},
\end{equation*}
\begin{equation*}
K_{1n} > -\frac{(\sec(\frac{\pi}{n}))^n\prod_{i=1}^{n}(-A_{ii})}{(\prod_{i=2}^{n}(A_{i(i-1)}))}-A_{1n}.
\end{equation*}
\end{enumerate}
In our special case, one of the above $2n$ conditions implies that the matrix with one nonzero term $K_{(10)(10)}=-9577.1955$ is one of the sparsest solutions. However, its corresponding quadratic cost is $10182.2224$ which is an extremely large number. In other words, such a solution gives an upper bound for the sparse solution corresponded to the extremely large value of $\alpha_1$ which is equivalent to the minimizing of $\|K\|_{\ell_0}$. It is very interesting that the sparsity level for both CVX and ADMM is equal to $10 \%$, however, it is equal to $1 \%$ for our previously mentioned cyclic system.
\section{Discussion and Conclusion}\label{Discon}
As it is shown, we propose a sub-optimal sparsity-promoting state feedback gain controller on the basis of affine approximation, Schur complement and SDP form. The numerical experiments show that our proposed method gives an appropriate approximation of optimal sparse solution of the $\ell_0$-regularized LQR problem. For $\alpha_1=0$ we get the standard LQR solution and by increasing it, we achieve more sparse state feedback gain controllers and the number of nonzero entries of state feedback gain controller is decreasing. And also, the optimal value is increasing. Thus, there is a fundamental trade-off between sparsity level and performance degradation.
An advantage of our proposed method is that just like algorithm $1$ in \cite{fardad2014design}, it can be solved by using CVX toolbox because of its convex SDP form. Although sparsifying rate of our proposed method for small values of $\alpha_1$ is higher compared to the proposed method in \cite{lin2013design}, the quadratic cost minimizing rate of the latter one is higher than ours. For large values of $\alpha_1$, such a behavior is getting reversed. Our proposed method is slower than the method presented in \cite{lin2013design}, however, it has a simpler form. Again, we insist on this point that the quadratic cost achieved by our proposed method is an upper bound for the exact optimal solution of the Sparsity-promoting LQR problem.

The improvement of the accuracy of our proposed method can be considered as a future work. Also, using the norms other than $\ell_1$-norm or utilizing the weighted version of the $\ell_1$-norm could result in better achievements. We try to use the Affine Approximation method as an appropriate tool in sparsity problems for which the sparsification occurs in time domain.

\appendix
The lemma \ref{prop1} is proved by using the following equivalent problems:\\
\begin{eqnarray}
& & \hspace{-0.6cm} \underset{X,K}{\textrm{Minimize}} ~\mathbf{Tr}(X) +\alpha_1 \|K\|_{\ell_1} \nonumber\\
& & \hspace{-0.6cm} \textrm{subject to:  } \nonumber\\
& & (A+BK)^TX+X(A+BK)+Q+K^TRK = 0,
\nonumber\\
& & \hspace{0.0cm} X \succ 0. \label{part1} 
\end{eqnarray}
\begin{eqnarray}
& & \hspace{-0.6cm} \underset{X,K}{\textrm{Minimize}} ~\mathbf{Tr}(X) +\alpha_1 \|K\|_{\ell_1} \nonumber\\
& & \hspace{-0.6cm} \textrm{subject to:  } \nonumber\\
& & (A+BK)^TX+X(A+BK)+Q+K^TRK \preceq 0,
\nonumber\\
& & \hspace{0.0cm} X \succ 0. \label{part2}
\end{eqnarray}
\begin{eqnarray}
& & \hspace{-1.05cm} \underset{X,K,F}{\textrm{Minimize}} ~\mathbf{Tr}(X) +\alpha_1 \|K\|_{\ell_1} \nonumber\\
& & \hspace{-1.05cm} \textrm{subject to:  } \nonumber\\
& & (A+BK)^TX+X(A+BK)+Q+F \preceq 0, \nonumber
\\
& & \hspace{0.0cm} X \succ 0, \nonumber\\ 
& & F \succeq K^TRK. \label{part3}
\end{eqnarray}
\begin{eqnarray}
& & \hspace{-0.45cm} \underset{X,K,F}{\textrm{Minimize}} ~\mathbf{Tr}(X) +\alpha_1 \|K\|_{\ell_1} \nonumber\\
& & \hspace{-0.45cm} \textrm{subject to:  } \nonumber\\
& & \hspace{-0.8cm} -Q-F \succeq \nonumber\\ 
& & \hspace{-0.8cm} \begin{bmatrix} X+\frac{(A+BK)^T}{2} & X-\frac{(A+BK)^T}{2}\end{bmatrix}\begin{bmatrix} I & 0\\ 0 & -I\end{bmatrix}\begin{bmatrix} X+\frac{A+BK}{2}\\ X-\frac{A+BK}{2}\end{bmatrix}, \nonumber
\\
& & \hspace{-0.8cm} X \succ 0, \nonumber\\
& & \hspace{-0.8cm} \begin{bmatrix} F & K^T\\ K & R^{-1}\end{bmatrix} \succeq 0. \label{part4}
\end{eqnarray}
\begin{eqnarray}
& & \hspace{-0.8cm} \underset{X,K,F,P}{\textrm{Minimize}} ~\mathbf{Tr}(X) +\alpha_1 \|K\|_{\ell_1} \nonumber\\
& & \hspace{-0.8cm} \textrm{subject to:  } \nonumber\\
& & \hspace{-0.8cm} -Q-F+
\begin{bmatrix} 2X-P^T & P^T \end{bmatrix}\begin{bmatrix} 0 & 0\\ 0 & (\delta+1)I\end{bmatrix}\begin{bmatrix} 2X-P^T\\ P\end{bmatrix} \nonumber\\
& & \hspace{-0.8cm} \succeq \nonumber\\
& & \hspace{-0.8cm} \begin{bmatrix} 2X-P^T & P^T \end{bmatrix}\begin{bmatrix} I & 0\\ 0 & \delta I\end{bmatrix}\begin{bmatrix} 2X-P\\ P\end{bmatrix},
\nonumber\\
& & \hspace{-0.8cm} P = X-\frac{A+BK}{2}, \nonumber\\
& & \hspace{-0.8cm} X \succ 0, \nonumber\\
& & \hspace{-0.8cm} \begin{bmatrix} F & K^T\\ K & R^{-1}\end{bmatrix} \succeq 0. \label{part5}
\end{eqnarray}
\begin{eqnarray}
& & \hspace{-3.2cm} \underset{X,K,F,P,Y}{\textrm{Minimize}} ~\mathbf{Tr}(X) +\alpha_1 \|K\|_{\ell_1} \nonumber\\
& & \hspace{-3.2cm} \textrm{subject to:  } \nonumber\\
& & \hspace{-3.2cm} -Q-F+Y \nonumber\\
& & \hspace{-3.2cm} \succeq \nonumber\\
& & \hspace{-3.2cm} \begin{bmatrix} 2X-P^T & P^T \end{bmatrix}\begin{bmatrix} I & 0\\ 0 & \delta I\end{bmatrix}\begin{bmatrix} 2X-P\\ P\end{bmatrix},
\nonumber\\
& & \hspace{-3.2cm} P = X-\frac{A+BK}{2}, \nonumber\\
& & \hspace{-3.2cm} X \succ 0, \nonumber\\
& & \hspace{-3.2cm} \begin{bmatrix} F & K^T\\ K & R^{-1}\end{bmatrix} \succeq 0, \nonumber\\
& & \hspace{-3.2cm} Y \succeq (\delta+1)P^T P, \nonumber\\
& & \hspace{-3.2cm} Y \preceq (\delta+1)P^T P. \label{part6}
\end{eqnarray}
\begin{eqnarray}
& & \hspace{-3cm} \underset{X,K,F,P,Y}{\textrm{Minimize}} ~\mathbf{Tr}(X) +\alpha_1 \|K\|_{\ell_1} \nonumber\\
& & \hspace{-3cm} \textrm{subject to:  } \nonumber\\
& & \hspace{-3cm} \begin{bmatrix} -Q-F+Y & 2X-P^T & P^T\\ 2X-P & I & 0\\P & 0 & \frac{1}{\delta} I\end{bmatrix} \succeq 0, \nonumber\\
& & \hspace{-3cm} P = X-\frac{A+BK}{2}, \nonumber\\
& & \hspace{-3cm} X \succeq \epsilon_1I, \nonumber\\
& & \hspace{-3cm} \begin{bmatrix} Y & P^T\\P & \frac{1}{\delta+1}I\end{bmatrix} \succeq 0, \nonumber\\
& & \hspace{-3cm} \begin{bmatrix} F & K^T\\ K & R^{-1}\end{bmatrix} \succeq 0, \nonumber\\
& & \hspace{-3cm} Y \preceq (\delta+1)P^TP. \label{part7}
\end{eqnarray}

The proof procedure from problem (\ref{part1}) to problem (\ref{part4}) is same as one which has been stated in \cite{fardad2014design}. We apply Schur Complement to problems (\ref{part3}) and (\ref{part6}) to achieve an SDP form in problems (\ref{part4}) and (\ref{part7}) respectively which makes them applicable to CVX. Although the first Schur complement is applied easily, some operations must be done during the problems (\ref{part4}) to (\ref{part6}) to apply the second Schur complement to part (\ref{part6}) to get problem (\ref{part7}).\\ In problem (\ref{part6}) we substitute $X \succ 0$ by $X \succeq \epsilon_1 I$ to get problem (\ref{part7}) where $\epsilon_1$ is sufficiently small number. We do such a substitution, because, it must have an SDP form in order to be applied to CVX. Also, in problem (\ref{part6}), $Y=(\delta+1)P^TP$ has been bifurcated to $Y \succeq (\delta+1)P^T P$ and $Y \preceq (\delta+1)P^T P$ where the Schur complement is applied to the first one to get the problem (\ref{part7}). We can do such a bifurcation. Because, they are equivalent conditions.
\section*{Acknowledgement}
The Author is thankful to Dr. Sadegh Bolouki for his fruitful comments.

\addtolength{\textheight}{-12cm}   

\begin{spacing}{.8}
\bibliographystyle{IEEEtran}
\bibliography{merg1}
\end{spacing}

\end{document}